\documentclass[12pt]{article}
\usepackage{amsmath,amsthm,amsfonts}

\newtheorem{thm}{Theorem}[section]

\newtheorem{lemma}[thm]{Lemma}
\newtheorem{prop}[thm]{Proposition}
\newtheorem{cor}[thm]{Corollary}

\newtheorem{defin}[thm]{Definition}
\newtheorem{definition}[thm]{Definition}
\newtheorem{example}[thm]{Example}
\newtheorem*{problem}{Open Problem}

\newtheorem{rem}[thm]{Remark}

\def\part{\partial}

\def\done{{1\hskip-2.5pt{\rm l}}}
\def\qed{\vrule height 1.2ex width 1.1ex depth -.1ex}

\newcommand{\Diff}{{\text{Diff}}}

\newcommand{\Ham}{{\text{Ham}}}

\newcommand{\Conj}{{\text{Conj}}}
\newcommand{\C}{{\text{Conj}}}
\newcommand{\supp}{{\text{support}}}

\newcommand{\R}{{\mathbb{R}}}

\newcommand{\T}{{\mathbb{T}}}
\newcommand{\Z}{{\mathbb{Z}}}
\newcommand{\N}{{\mathbb{N}}}

\newcommand{\cG}{{\mathcal{G}}}

\newcommand{\id}{{\text{{\bf 1}}}}

\begin{document}

\title{Conjugation-invariant norms on groups of geometric origin\\
 }

\renewcommand{\thefootnote}{\alph{footnote}}

\author{\textsc Dmitri Burago$^{a}$,\ Sergei Ivanov$^{b}$\ and Leonid Polterovich$^{c}$}

\footnotetext[1]{Partially supported by the NSF Grant DMS-0412166. }
\footnotetext[2]{Partially supported by  RFBR grant 05-01-00939.}
\footnotetext[3]{Partially supported by the Israel Science
Foundation grant $\#$ 509/07.}

\date{\today}

\maketitle

\begin{abstract}
\noindent A group is said to be bounded if it has a finite diameter
with respect to any bi-invariant metric. In the present paper we
discuss boundedness of various groups of diffeomorphisms.
\end{abstract}

\bigskip

\section{Introduction and main results}

\subsection{The main phenomenon}

A group $G$ is said to be {\it bounded} if it is bounded
with respect to any bi-invariant metric (that is, as a metric space,
it has a finite
diameter).

A {\it conjugation-invariant norm} $\nu: G \to [0;+\infty)$ is a
function which satisfies the following axioms:
\begin{itemize}
\item[{(i)}] $\nu(1) = 0$;
\item[{(ii)}] $\nu(f) = \nu(f^{-1}) \;\; \forall f \in G$;
\item[{(iii)}] $\nu(fg) \leq \nu(f)+\nu(g) \;\; \forall f,g \in G$;
\item[{(iv)}] $\nu(f) = \nu (gfg^{-1}) \;\; \forall f,g \in G$;
\item[{(v)}] $\nu(f) > 0$ for all $f\neq 1$.
\end{itemize}

Thus a group is bounded iff every conjugation-invariant norm is
bounded.

\medskip
\noindent{\bf Convention: In this paper we work only with
conjugation-invariant norms, so by default a {\it norm} is a
conjugation-invariant norm.}

\medskip
\noindent If one drops condition (v), $\nu$ is said to be a {\it
pseudo-norm}. It can immediately be converted into a norm by adding
$1$ to all elements except the unity. Hence a group is unbounded if
it admits an unbounded pseudo-norm. Observe that on a simple group
every non-trivial pseudo-norm is automatically a norm: Indeed, the
set of all elements with vanishing pseudo-norm forms a normal
subgroup. Hence in the sequel condition (v) can be dropped
everywhere when we deal with simple groups such as groups of smooth
diffeomorphisms.

\medskip
\noindent Two norms on a group are called {\it equivalent} if their
ratio is bounded away from $0$ and $\infty$. {\it The trivial norm},
which exists on any group, equals $1$ on every element except the
identity.

Given a connected manifold $M$, denote by $\Diff_0(M)$ the identity
component of the group of $C^\infty$ smooth compactly supported
diffeomorphisms. The central phenomenon discussed in this paper is
as follows: {\it in all known to us examples  any norm on
 $\Diff_0(M)$ is equivalent to the trivial one.} Below we confirm this
phenomenon for spheres, all closed connected three-manifolds and the
annulus. However we have neither a proof nor a counter-example for
closed surfaces of genus $\geq 1$ and the M\"{o}bius strip.

\subsection{Setting the stage}

\subsubsection{Conjugation-generated norms}
Many interesting norms come from the following construction: Let $G$
be a group. We say that a set $K\subset G$ {\it
conjugation-generates} (or, for brevity, {\it c-generates}) $G$  if
every element $h\in G$ can be represented as a product
\begin{equation}
\label{eq-decomp-0}h={\tilde h_1} {\tilde h_2}\dots {\tilde h_N}
\end{equation}  where each $\tilde h_i$ is conjugate to some element $h_i \in
K$: ${\tilde h_i}=\alpha_i h_i \alpha_i^{-1}$, $\alpha_i \in G$. In
this case define a norm $q_K(h)$ as the minimal $N$ for which such a
representation exists. We shall say that {\it the norm $q_K$ is
c-generated by the subset $K$.} If $K$ is finite, $G$ is said to be
{\it finitely c-generated}. For instance, every simple group $G$ is
finitely $c$-generated by $K=\{x,x^{-1}\}$ with an arbitrary $x \neq
1$.

Note that the norm $q_K$ has the following extremal property: for
any norm $q$ bounded on $K$ there is a constant $\lambda$ such that
$q \leq \lambda q_K$. Hence, if $K$ is finite, the group $G$ is
bounded if and only if $q_K$ is bounded.

\medskip
\noindent
\begin{example}\label{exam-SL}{\rm
Groups $SL(n,\R)$ for $n \geq 2$ and $SL(n,\Z)$ for $n \geq 3$ are
finitely c-generated by the set $K$ of all elementary matrices whose
off-diagonal term equals $\pm 1$. Moreover we claim that the number
of terms in the decomposition \eqref{eq-decomp-0} is bounded by a
constant which does not depend on $h$.

In the case of $SL(n,\R)$ the claim follows from an appropriate
version of the Gauss elimination process.

As for $SL(n, \Z)$,  denote by $\cal{E}$ the set of all elementary
matrices whose only non-zero off-diagonal element equals to $1$.
There exists $N=N(n) \in \N$ so that every element from $SL(n,\Z)$
can be written as a product of $\leq N$  matrices of the form $E^p$,
where $E \in \cal{E}$ and $p \in \Z$ (in other words, $SL(n,\Z)$
possesses a bounded generation by elements from $\cal{E}$), see
\cite{SL}. The claim readily follows from the fact that each
$E^p=[A,B^p]$ for some $A,B \in \cal{E}$. Let us prove this
identity: let $E_{ij}$ (where $i\ne j$) denotes the elementary
matrix from $\cal{E}$ whose only non-zero off-diagonal element
stands in the $i$-th raw and $j$-th column. Without loss of
generality, put $i=1,j=3$. Then $E_{13}^p=[E_{12},E_{23}^p]$ as
required.

It follows from the claim that the "extremal" norm $q_K$ is bounded,
and hence the groups in question are bounded in view of extremality
of $q_K$. }
\end{example}

\medskip
\noindent
\begin{example}
\label{commnorm} {\rm {\bf The commutator length.} Given a group
$G$, denote by $G'$ its commutator subgroup. The norm on $G'$
c-generated by the set of all simple commutators
$[a,b]=aba^{-1}b^{-1}$ is called the {\it commutator length} and is
denoted by $cl_G$.}
\end{example}

\subsubsection{The role of the commutator
subgroup}\label{subsubsec-role} The next observations suggest that
the commutator subgroup plays a significant role in the study of
boundedness.

\begin{prop}\label{prop-abel}
If $H_1(G):=G/G'$ is infinite then $G$ is unbounded.
\end{prop}

\medskip
\noindent In particular, {\it an abelian group is bounded if and
only if it is finite.}

\medskip

Note that unbounded norms maybe non-extendable from a normal
subgroups to the ambient  group. Consider, for instance, the
extension $Aff(\Z)$ of $\Z$ by an element $t$ of order 2 and with
one additional relation $tz=z^{-1}t$. Thus $\Z$ is a normal subgroup
of index 2 in $Aff(\Z)$. Of course, $\Z$ has an unbounded norm,
while $Aff(\Z)$ admits no unbounded norms since $t$ is conjugate to
$tz^{2n}$ (by $z^n$) for all integers $n$. However, the situation
changes when one deals with the commutator length on the commutator
subgroup:

\begin{prop}\label{comm-length-bdd} Let $G$ be any group.
If the commutator length on $G'$ is unbounded then $G$ itself is
unbounded.
\end{prop}

\medskip
\noindent Propositions \ref{prop-abel} and \ref{comm-length-bdd} are
proved in Section \ref{subsec-qnorms} below.

\subsubsection{Stably unbounded norms} Given a conjugation-invariant
norm $\nu$ on a group $G$, we define its {\it stabilization} by
$$\nu_{\infty} (f)= \lim_{n \to \infty}\frac{\nu(f^n)}{n}\;.$$ Let
us emphasize that stabilization of a norm is not in general a norm.
An unbounded norm $\nu$ is called {\it stably unbounded} if
$\nu_{\infty}(f) \neq 0$ for some $f \in G$.

For instance, an infinite abelian torsion group is unbounded by
Proposition \ref{prop-abel} but never stably unbounded.

\begin{example}\label{ex-discr-volsup} {\rm
Consider a group $\Z_2^\infty$ of all finite words over $\{0,1\}$
with componentwise addition mod 2 (that is, a direct product of
countably many copies of $\Z_2$). This group admits no
quasi-morphisms since the order of every element is 2. On the other
hand, the length of a word is a norm. There is a natural action of
$\Z_2^\infty$ on $\Z \times \Z_2$: the $i$-th generator swaps
$(i,0)$ and $(i, 1)$. Thus the norm in our example can be
interpreted as "the size of support".}
\end{example}

\begin{problem}\label{quest-stab} {\rm Does there exist a group that does not admit a
stably unbounded norm and yet admits a norm unbounded on some cyclic
subgroup?}
\end{problem}

\subsubsection{Stable commutator length and
quasi-morphisms}\label{subsec-quasim} In what follows we shall focus
on the stable commutator length. Let $G$ be any group. The
commutator length $cl_G$ on $G'$ is stably unbounded if and only if
$G$ admits {\it non-trivial homogeneous quasi-morphisms}
\cite{Bavard}. Recall that a function $r:G \to R$ is called a
quasi-morphism if there exists $C>0$ so that
$$|r(ab)-r(a)-r(b)| \leq C \;\;\forall a,b \in G\;.$$
A quasi-morphism is called {\it homogeneous} if $r(a^n)=nr(a)$ for
all $a\in G$ and $n \in \Z$.  A quasi-morphism is called non-trivial
if it is not a morphism.

\medskip
\noindent{\bf Convention: In this paper we deal with homogeneous
quasi-morphisms only, so by default quasi-morphism means a
homogeneous quasi-morphism.}
\medskip
\noindent

\begin{example}{\rm
$G=SL(2,\Z)$ carries an abundance of quasi-morphisms (cf.~ e.g.
\cite{Barge-Ghys}) and hence the commutator norm on $SL(2,\Z)$ is
stably unbounded. Thus $G$ is unbounded in view of Proposition
\ref{comm-length-bdd}, in contrast with $SL(n,\Z)$ for $n \geq 3$
(see Example \ref{exam-SL} above).}
\end{example}

\medskip
\noindent Introduce the class $\cG$ of groups $G$ with finite
$H_1(G)= G/G'$ (we wish to rule out  conjugation-invariant stably
unbounded norms coming from the first homology, see Proposition
\ref{comm-length-bdd} above). Note that various interesting groups
of diffeomorphisms are simple and hence belong to this class.

\begin{problem}\label{quest-stab-1} {\rm Does there exist a
finitely presented group $G \in \cG$ whose commutator length is
unbounded but stably bounded?}
\end{problem}

\begin{problem}{\rm Does there exist an unbounded finitely presented
group which admits no unbounded quasi-morphisms? }
\end{problem}

\medskip
\noindent A.~Muranov informed us that he has an example of a
finitely generated, but not finitely presented, group from $\cG$
whose commutator length is unbounded but stably bounded. The
existence of an infinitely generated group with the this property
readily follows from Muranov's work \cite{Muranov}, who constructed
a sequence of simple groups $G_i, i \in \N$ of finite commutator
length diameter $n_i$, where $n_i \to \infty$. The infinite direct
product $G=\prod_i G_i$ is as required.

\medskip
\noindent A  mystery related to the notion of stable unboundedness
is as follows.

\begin{problem}\label{quest-stab-3} {\rm Does there exist a group $G \in \cG$
whose commutator length is stably bounded, but which admits a
stably unbounded norm?  In other words, does the existence of a
stably unbounded norm on $G$ yields existence of non-trivial
quasi-morphisms? In fact, we do not know even a single example of
a group from $\cG$ that admits no non-trivial quasi-morphisms but
carries a norm that is unbounded on some cyclic subgroup.}
\end{problem}

\medskip
\noindent Here is a (somewhat artificial) example of groups for
which existence of a stably unbounded norm yields existence of
non-trivial quasi-morphisms. Start with an arbitrary group $G \in
\cG$ and set $\bar G$ to be the extension of $G$ by an element $t$
so that
$$t^2=1,\;\;\;\text{and}\;\;\; t(g_1,g_2)t^{-1} =(g_2,g_1) \;\; \forall g_1,g_2 \in G.$$

\begin{prop}\label{prop-gbar} The group $\bar G$ lies in $\cG$
for every $G \in \cG$.
\end{prop}

\begin{prop}\label{thm-gbar} Suppose that for some $G \in \cG$, the group $\bar G$ admits a stably unbounded
norm. Then $\bar G$ admits a non-trivial quasi-morphism.
\end{prop}

\subsubsection{Quasi-norms}\label{subsec-qnorms}
\begin{definition}
\label{qnorm} Let $G$ be a group. We say that a function $q: G \to
[0;+\infty)$ is a {\it a quasi-norm} (for brevity, a {\it q-norm})
if:
\begin{itemize}
\item[{(i)}] $q$ is quasi-subadditive: there is a constant $c$ such that
$$q(ab) \leq q(a)+q(b)+c\;;$$
\item[{(ii)}] $q$ is quasi-conjugation-invariant: there is a constant $c$ such
that $$|q(b^{-1}ab) - q(b)| \leq c\;;$$
\item[{(iii)}] $q$ is unbounded.
\end{itemize}
\end{definition}

One can see that in fact the existence of a q-norm implies the
existence of an unbounded norm: This norm can be constructed by (i)
symmetrization: taking the maximum of the norm of $a$ and $a^{-1}$
for each $a$, (ii) redefining the norm of $a$ to be the maximum of
norms of its conjugates $b^{-1}ab$, and (iii) by adding a
sufficiently large constant to the norm of all elements excluding
the identity.

{\it Hence a group is unbounded if it admits a q-norm; in other
words, the existence of unbounded norms and q-norms are equivalent.}
However q-norms are often defined in a more natural way: A
motivating example is provided by the absolute value of a
non-trivial homogeneous quasi-morphism. Another advantage of q-norms
is that they behave nicely under epimorphisms:

\begin{lemma}
\label{basic-epi} The pull-back of a q-norm under an epimorphism is
a q-norm. In particular, if a group $G$ admits a homomorphism onto
an unbounded group, $G$ itself is unbounded.
\end{lemma}

\medskip
\noindent This follows immediately from the definitions and
discussion above. Let us apply the lemma for proving results stated
in \ref{subsubsec-role}:

\medskip
\noindent{\bf Proof of Proposition \ref{prop-abel}:}

\medskip
\noindent{\sc Step 1:} Let us show that any infinite abelian group
$G$ admits an unbounded norm.

If $G$ is finitely generated, than by the classification theorem it
has a $\Z$ as a direct factor, and hence it admits an epimorphism
onto $\Z$. Thus $G$ admits an unbounded norm by Lemma
\ref{basic-epi}.

For a countably generated $G$, let us enumerate its generators $g_1,
g_2, \dots$. Define the norm of $g$ to be the smallest $k$ such that
$g$ lies in the subgroup generated by $g_1, g_2, \dots , g_k$. This
norm is unbounded.

In general, any infinite abelian group contains an infinite finitely
or countably generated subgroup, and the above construction provides
us with a norm on this subgroup $H$. Now choose any element $g$ from
$G \setminus H$ and consider a subgroup $H'$ generated by the union
of $H$ and $g$. Combining the easily verifiable fact that the norm
extends from $H$ to $H'$ with Zorn's lemma completes the proof.

\medskip
\noindent{\sc Step 2:} Assume now that $G/G'$ is infinite. By Step
1, it admits an unbounded norms. Look at the epimorphism $G \to
G/G'$. Applying Lemma \ref{basic-epi} we conclude that $G$ is
unbounded. \qed

\medskip
\noindent {\bf Proof of Proposition \ref{comm-length-bdd}:} If
$[G,G]$ has infinite index, look at the epimorphism $G \to H:=G/G'$.
The group $H$ is an infinite abelian group, thus by Proposition
\ref{prop-abel} $H$ is unbounded, and hence $G$ is unbounded in view
of Lemma \ref{basic-epi}.

Otherwise, if $H$ is finite,  one can check that the commutator norm
can be extended from the commutator to the whole group (even though
in general q-norms cannot be extended from finite index subgroups,
see an example above). Indeed, pick a (finite!) set $S$ of
representatives from cosets of $G'$. Then every element of $G$ can
be uniquely written as $hs$ where $h\in G'$, $s\in S$. Define a
q-norm of such an element $g=hs$ by $q(g)=cl_G(h)$. The approximate
conjugation invariance of this norm follows from the fact that
conjugation can be written as a multiplication by a commutator (and
hence it changes the norm by at most~1). To prove the approximate
triangle inequality, note that for $g_1=h_1s_1$ and $g_2=h_2s_2$
$$g_1g_2 = h_1h_2[h_2^{-1},s_1]s_1s_2\;.$$
Write $$s_1s_2 = h(s_1,s_2)t(s_1,s_2)\;,$$ where $h(s_1,s_2) \in G'$
and $t(s_1,s_2) \in S$. Thus
$$||g_1g_2|| = cl_G(h_1h_2[h_2^{-1},s_1]h(s_1,s_2))\;.$$
Put $C= \max_{s_1,s_2 \in S} cl_G(h(s_1,s_2))$. Applying the
triangle inequality for the commutator length, we get
$$||g_1g_2|| \leq cl_G(h_1) + cl_G(h_2) + 1 + C =
||h_1||+||h_2||+1+C\;.$$ Thus $q$ is indeed a q-norm. \qed

\subsubsection{Fine norms}
A norm $\nu$ on $G$ is called {\it fine} if $0$ is a limit point of
$\nu(G)$. Otherwise the norm is called, following a suggestion by
Yehuda Shalom, {\it discrete}. For instance, conjugation-generated
norms assume integer values only and hence are discrete. On the
other hand a bi-invariant Riemannian metric on a compact Lie group
gives rise to a bounded fine norm on the group.

\subsubsection{Meager groups}
A norm $\nu$ on a group is {\bf not} equivalent to the trivial norm
if it is either unbounded  or fine. A group $G$ is called {\it
meager} if every conjugation-invariant norm on $G$ is equivalent to
the trivial one (i.e. is bounded and discrete).

\subsection{Norms on diffeomorphism groups}

\subsubsection{Smooth diffeomorphisms}
In this section we present the main results of the paper which deal
with norms on groups $\Diff_0(M)$, where $M$ is a smooth connected
manifold. We start with the case of closed manifolds.

\begin{thm}[Main Theorem]\label{thm-1}
 $\;$\begin{itemize}
\item[{(i)}] The group $\Diff_0(M)$ does not admit a fine conjugation-invariant
norm for all connected manifolds $M$.
\item[{(ii)}] The group $\Diff_0(S^n)$ is meager (where $S^n$ is a
sphere);
\item[{(iii)}] The group $\Diff_0(M)$ is meager for any
closed connected 3-dimensional manifold $M$.
\end{itemize}
\end{thm}

Let us give two important examples of conjugation-invariant norms on
$\Diff_0(M)$.

\begin{example} \label{ex2} {\rm {\bf The commutator length:}
Since $\Diff_0(M)$ is a simple group (see e.g. \cite{Banyaga}) it
coincides with its commutator subgroup and hence the commutator
length (see Example \ref{commnorm}) is a well-defined invariant norm
on $\Diff_0(M)$. Introduce the {\it commutator length diameter}
$cld(M) \in \N \cup \infty$ as $\max cl(f)$ over all $f \in
\Diff_0(M)$.}
\end{example}

\bigskip

\begin{thm}\label{thm-uni-perf}
$\;$\begin{itemize}
\item[{(i)}] For the sphere, $cld(S^n) \leq 4$;
\item[{(ii)}] For any closed connected 3-dimensional manifold $M$,
$cld(M) \leq 10$.
\end{itemize}
\end{thm}

\begin{example} \label{ex2-1} {\rm {\bf The fragmentation norm:}
Every element $f \in \Diff_0(M)$ can be represented as a finite
product of diffeomorphisms supported in an embedded open ball (this
is the famous fragmentation lemma, see e.g. \cite{Banyaga}). The
{\it fragmentation norm} $frag(f)$ is the minimal number of factors
required to represent an element $f \in \Diff_0(M)$. Clearly, $frag$
is an conjugation-invariant norm on $\Diff_0(M)$. The next result
shows that the fragmentation norm is responsible for meagerness of
$\Diff_0(M)$. }
\end{example}

\begin{prop}\label{thm-frag} The group $\Diff_0$ is meager if and
only if the fragmentation norm is bounded.
\end{prop}

\begin{problem}{\rm
Is the fragmentation norm is bounded for the case of closed
surfaces?}
\end{problem}

\medskip

Let us now turn to open manifolds.

\medskip
\noindent
\begin{defin}\label{def-portable}{\rm
We say that a smooth connected open manifold $M$ is {\it
portable}\footnote{This notion is a mock version of subcritical
Liouville manifolds in symplectic topology.} if it admits a complete
vector field $X$ and a compact subset $M_0$ with the following
properties:
\begin{itemize}
\item $M_0$ is an attractor of the flow $X^t$ generated by $X$:
for every compact subset $K \subset M$ there exists $\tau>0$ so that
$X^{\tau}(K) \subset M_0$.
\item There exists a diffeomorphism $\theta \in \Diff_0(M)$ so that
$\theta(M_0) \cap M_0 = \emptyset$.
\end{itemize}
The set $M_0$ is called {\it the core} of a portable manifold $M$. }
\end{defin}

\medskip
\noindent For instance, any manifold $M$ which splits as $P \times
\R^n$, where $P$ is a closed manifold, is portable. Indeed, the
vector field $X(p,z)=-z\frac{\partial}{\partial z}$ and the compact
$M_0= P \times \{|z|\leq 1\}$ satisfy the conditions above.
Furthermore, $M$ is portable if it admits an exhausting Morse
function with finite number of critical points so that  all the
indices are strictly less than $\frac{1}{2}\dim M$. This implies,
for example, that every 3-dimensional handlebody is a portable
manifold.

\medskip
\noindent The next result is the main "local" block in the proof of
Theorem \ref{thm-1}(ii) and (iii).

\begin{thm}\label{thm-portable} The group $\Diff_0(M)$ is meager provided $M$ is
portable.
\end{thm}

\medskip
\noindent For instance, any norm on $\Diff_0$ of an open ball is
bounded. This immediately yields Proposition \ref{thm-frag}.
Furthermore, $\Diff_0$ of a 2-dimensional annulus is meager (as well
as for any product $\R\times M$). However, it is still unknown
whether the same holds for the open M\"{o}bius band!

\medskip
\noindent Our next result deals with the commutator length diameter
of a portable manifold.

\begin{thm}\label{thm-uni-perf-port}
For a portable manifold $M$, $cld(M) \leq 2$.
\end{thm}

\subsubsection{Volume-preserving and symplectic
diffeomorphisms: examples and problems }

\medskip
\noindent In contrast to groups $\Diff_0$, the identity components
of groups of compactly supported volume preserving and symplectic
diffeomorphisms, as well as their commutator subgroups, are {\bf
never} meager: they admit a fine norm.

\begin{example}\label{ex-vol-support} {\rm {\bf The size-of-support norm:}
The counterpart of Example \ref{ex-discr-volsup} above for
diffeomorphism groups is as follows. Consider the identity component
$\Diff_0(M,\text{vol})$ of the group of compactly supported
volume-preserving diffeomorphisms of a smooth manifold $M$ of
dimension $> 0$. Define a norm of a diffeomorphism as the volume of
its support. This norm is necessarily fine, and it is unbounded
whenever the volume of $M$ is infinite. However this norm is never
stably unbounded: in fact, it is bounded on all cyclic subgroups.}
\end{example}

\medskip
\noindent In some situations, stably unbounded norms on the
commutator subgroup of $\Diff_0(M,\text{vol})$ can be "induced" from
the fundamental group of $M$ even when the volume of $M$ is finite:

\begin{example}\label{ex-sb-4}{\rm Suppose that
$M$ is a closed manifold equipped with a volume form.  Suppose that
$H:=\pi_1(M)$ has trivial center. Then the commutator length on the
commutator subgroup of $\Diff_0(M,\text{vol})$ is stably unbounded
provided the commutator length on $H'$ is stably unbounded, see
\cite{Ghys-Gambaudo,PMontreal}.}\end{example}

\medskip
\noindent However, no unbounded norms on volume-preserving
diffeomorphisms are known so far in the cases when the manifold has
simple topology and finite volume.

\begin{problem} \label{quest-volum} {\rm Assume that $n \geq 3$.
Does the identity component of the group of volume preserving
diffeomorphisms of the sphere $S^n$ admit an unbounded
conjugation-invariant norm? Does the identity component of the group
of compactly supported volume preserving diffeomorphisms of the ball
of finite volume admit an unbounded conjugation-invariant norm?}
\end{problem}

\medskip
\noindent In the symplectic category, interesting norms inhabit the
group $\Ham(M,\omega)$ of compactly supported Hamiltonian
diffeomorphisms of a symplectic manifold $(M,\omega)$.

\begin{example}{\rm {\bf The Hofer norm} on  $\Ham(M,\omega)$ (see e.g. \cite{Pbook})
is fine. Its unboundedness is a long-standing conjecture in
symplectic topology. Nowadays it is confirmed for various symplectic
manifolds including for instance surfaces, complex projective spaces
with the Fubini-Studi symplectic form and closed manifolds with
$\pi_2=0$. Further, the Hofer norm on groups of Hamiltonian
diffeomorphisms is known to be stably unbounded for various closed
symplectic manifolds. However it is unbounded, but not stably
unbounded, for the standard symplectic vector space $\R^{2n}$
(Sikorav, \cite{Sikorav}).}
\end{example}

\begin{example}\label{ex-sb-5}{\rm  The commutator length on $\Ham(M,\omega)$ is
known to be stably unbounded for various closed symplectic manifolds
(see \cite{Barge-Ghys, Entov, Entov-Polterovich, Ghys-Gambaudo,
Py-1,Py-2}) including all surfaces and complex projective spaces of
arbitrary dimension.}
\end{example}

\begin{example}
\label{compsupportsympl} {\rm The group $\Ham(\R^{2n})$  admits the
Calabi homomorphism (the average Hamiltonian) to $\R$. The kernel of
the Calabi homomorphism coincides with the commutator subgroup of
$\Ham(\R^{2n})$, which is known to be simple \cite{Banyaga}. This
group is stably bounded with respect to the commutator length. This
is proved by D.~Kotschick in \cite{Kotschick}. Alternatively, this
readily follows from the algebraic packing inequality given by
Theorem \ref{thm-algpack} below. In contrast to this, the commutator
length on $[\Ham(B^{2n}),\Ham(B^{2n})]$, where $B^{2n}$ is the
standard symplectic ball, is stably unbounded, see \cite{BEP}. }
\end{example}

\medskip
\noindent\begin{example}{\rm A somewhat less understood example is
{\it the fragmentation norm} (cf. Example \ref{ex2-1} above). Let
$(M,\omega)$ be a closed symplectic manifold and let $U \subset M$.
The Hamiltonian fragmentation lemma (see \cite{Banyaga}) states that
every Hamiltonian diffeomorphism $f$ can be written as a product
$h_1\circ...\circ h_N$, where each $h_i$ is conjugate to an element
from $\Ham(U)$. Define the fragmentation norm $frag_U(f)$ as the
minimal number of factors in such a decomposition. Using methods of
\cite{EP-qst}, one can show that $frag_U$ is unbounded on
$\Ham(\T^{2})$ provided the subset $U$ is displaceable by a
Hamiltonian diffeomorphism (e.g. $U$ is a ball of a small diameter).
Indeed combining Theorem 7.1 in \cite{EP-qst} with the fact that the
group $\Ham(\T^{2})$ is simply connected  one gets that
\begin{equation}\label{eq-Floer}
|\mu(\phi\psi)-\mu(\phi)-\mu(\psi)| \leq \min
(frag_U(\phi),frag_U(\psi)) \end{equation} for all $\phi,\psi \in
\Ham(\T^{2})$, where $\mu$ is the appropriate asymptotic spectral
invariant (we refer to \cite{Schwarz, Oh-spectral, MS2} for
preliminaries on spectral invariants). Take a pair of disjoint
meridians $L$ and $K$ on the torus. Let $\Phi,\Psi$ be two smooth
cut off functions on the torus with disjoint supports which equal
$1$ near $L$ and $K$ respectively. Let $\{\phi_t\}$ and $\{\psi_t\}$
be the Hamiltonian flows generated by $\Phi$ and $\Psi$. A standard
calculation in Floer homology shows that the left hand side of
\eqref{eq-Floer} with $\phi=\phi_t,\psi=\psi_t$ goes to infinity as
$t \to \infty$. This proves unboundedness of the Hamiltonian
fragmentation norm $frag_U$ on for the 2-torus. For
higher-dimensional tori, as it was pointed out to us by D.~McDuff,
spectral invariants are still well defined on $\Ham$ due to a result
by M.~Schwarz \cite{Schwarz}, and thus the argument above goes
through. However the question on unboundedness of the fragmentation
norm is still open, for instance, for the complex projective spaces
in any dimension. }\end{example}

\bigskip
\noindent{\sc Organization of the paper:} In the next section we
introduce algebraic packing and displacement technique which is used
for the proof of the main results stated in the introduction. As an
illustration, we deduce there Theorem \ref{thm-1}(i) and Proposition
\ref{thm-gbar}.  Theorems \ref{thm-portable} and
\ref{thm-uni-perf-port} are proved in Section \ref{subsec-to-port}.
These  theorems, combined with topological decomposition technique
(which is standard in the case of spheres, and less trivial in the
case of three-manifolds) is applied to the proof of Theorems
\ref{thm-1}(ii),\ref{thm-uni-perf}(i) in Section \ref{subsec-to-sph}
and of Theorems \ref{thm-1}(iii),\ref{thm-uni-perf}(ii) in Section
\ref{subsec-to-threeman}.

\section{Algebraic tools: packing and displacement}

Here we present the algebraic tools used for proving Theorems
\ref{thm-1}(i), \ref{thm-portable} and \ref{thm-uni-perf-port}. We
use a number of tricks which imitate displacement of supports of
diffeomorphisms and decomposition of diffeomorphisms into products
of commutators in a more general algebraic setting. The tricks of
this nature appear in the context of transformation groups at least
since the beginning of 1960-ies (see e.g. \cite{A}). The system of
notions introduced below in parts imitates and extends the one
arising in the study  of Hofer's geometry on the group of
Hamiltonian diffeomorphisms. Note also that various interesting
results on infinitely displaceable subgroups  were obtained in a
recent work of D.~Kotschick \cite{Kotschick}.

\subsection{Algebraic packing and displacement energy}

Let $G$ be any group. We say that two subgroups $H_1,H_2 \subset
G$ commute if $h_1h_2=h_2h_1$ for all $h_1 \in H_1, h_2 \in H_2$.
We denote by $Conj_\phi$ the automorhism of $G$ given by $g
\mapsto \phi g\phi^{-1}$. A subgroup $H \subset G$ is called
$m$-{\it displaceable} (where $m \geq 1$ is an integer) if there
exist elements $\phi_0:=1,\phi_1,...,\phi_{m} \in G$ so that the
subgroups $\Conj_{\phi_i}(H),\Conj_{\phi_j}(H) $ pair-wise commute
for all distinct $i,j \in \{0;...;m\}$. A subgroup $H$ is called
{\it strongly} $m$-displaceable if in the previous definition one
can choose $\phi_k$'s to be consecutive powers of the same element
$\phi \in G$: $\phi_k = \phi^k$. In this case we shall say that
$\phi$ $m$-displaces $H$.

Note that for $m=1$ both notions coincide, and, for brevity, we
refer to a $1$-displaceable subgroup as to displaceable.

Introduce two numerical invariants related to the above notions.
The {\it algebraic packing number} $p(G,H)=m+1$, where $m$ is the
maximal integer such that $H$ is $m$-displaceable. This is a
purely algebraic invariant. The second quantity involves a
conjugation-invariant norm, say $\nu$ on $G$. Define the {\it order $m$
displacement energy} of $H$ with respect to $\nu$ as $e_m(H) =
\inf \nu(\phi)$ where the infimum is taken over all $\phi \in G$
which $m$-displace $H$. We put $e_m(H)=+\infty$ if $H$ is not
strongly $m$-displaceable.

While speaking on displaceability, we tacitly assume that the
subgroup $H$ is non-abelian. Indeed, every abelian subgroup $H$ is
$m$-displaceable by $1$ for every $m \in \N$ and hence $e_m(H)=0$.

\medskip
\noindent
\begin{example}\label{ex-disp-diffeo}{\rm Let $M$ be a smooth connected
manifold. Put $G = \Diff_0(M)$. Take any open ball $B \subset M$.
Let $H$ be the subgroup of $G$ consisting of all diffeomorphisms
supported in $B$. Choose any diffeomorphism $\phi \in \Diff_0(M)$
which displaces $B$: $B \cap \phi(B) = \emptyset$. Then $H$
commutes with $Conj_{\phi}(H)$, so $H$ is displaceable. }
\end{example}

\medskip
\noindent
\begin{thm}\label{thm-master} Let $H \subset G$ be a strongly
$m$-displaceable subgroup of $G$. Assume that $G$ is endowed with
a conjugation-invariant norm $\nu$.
\begin{itemize}
\item[{(i)}] For every element $x \in H'$ with $cl_H(x)=m$ the following
inequalities hold:
\begin{equation}\label{eq-master}
\nu(x) \leq 14 e_m(H)\;
\end{equation}
and
\begin{equation}\label{eq-master-com}
cl_G(x) \leq 2;
\end{equation}
\item[{(ii)}] In the case $cl_H(x)=1$, that is $x=[f,g]$ for some $f,g \in H$, we have that
\begin{equation}\label{eq-master-1}
\nu(x) \leq 4e_1(H)\;;
\end{equation}
\end{itemize}
\end{thm}

\medskip
\noindent
\begin{cor}\label{cor-port}
Assume that an element $F \in G$ $m$-displaces $H$ for {\it every}
$m \geq 1$. Then  $cl_G(h) \leq 2$  for all $h \in H'$.
\end{cor}

\medskip
\noindent This follows immediately from inequality
\eqref{eq-master-com}.

\medskip
\noindent Theorem \ref{thm-master}(ii) is proved in \cite{EP}. The
argument is very short: indeed, assume that $ \Conj_{\phi}(H)$
commutes with $H$. Then $[f,g] = [f \cdot \phi f^{-1}
\phi^{-1},g]$. Using bi-invariance of $\nu$ we get that
$$\nu([f,g]) \leq 2\nu([f,\phi]) \leq 4\nu(\phi)\;.$$
Taking the infimum over all $\phi$ displacing $H$ we get inequality
\eqref{eq-master-1}. The proof of Theorem \ref{thm-master}(i) is
more involved, see Section \ref{sec-ineq-com}  below.

\medskip
\noindent \begin{rem}{\rm For each pair of subgroups $H_1,H_2
\subset G$ one can define the "disjunction energy" $e(H_1,H_2)$ as
the infimum of $\nu(\phi)$ where $H_1$ commutes with
$\Conj_{\phi}(H_2)$. The argument above shows that
$\nu([h_1,h_2])\leq 4e(H_1,H_2)$.}
\end{rem}

\medskip
Let us give some sample applications of Theorem \ref{thm-master}.
First, we deduce from inequality \eqref{eq-master-1} the fact that
the group $\Diff_0(M)$ does not admit a fine norm.

\medskip
\noindent{\bf Proof of Theorem \ref{thm-1}(i):} Assume on the
contrary that $\Diff_0(M)$ admits a fine norm, say $\nu$. Take any
ball $B \subset M$ and pick two non-commuting diffeomorphisms $f$
and $g$ supported in $B$. For any $\epsilon >0$ take $h \in
\Diff_0(M)$ with $0< \nu(h) < \epsilon$. Note that since $h \neq
\done$ there exists a ball $C \subset M$ so that $h$ displaces
$C$. Since all balls in $M$ are isotopic, there is a
diffeomorphism $\psi \in \Diff_0(M)$ with $\psi(C)=B$. Therefore
$\phi:=\psi h \psi^{-1}$ displaces $B$, and hence $\phi$ displaces
the subgroup $\Diff_0(B) \subset \Diff_0(M)$. Applying inequality
\eqref{eq-master-1} we get that
$$\nu([f,g]) \leq 4\nu(\phi) = 4\nu(h) < 4\epsilon\;.$$
Sending $\epsilon$ to zero, we conclude that $\nu([f,g])=0$, a
contradiction with the non-degeneracy of a norm. \qed

\medskip

Next, we apply Theorem \ref{thm-master} to proving that for a class
of groups introduced in Section \ref{subsec-quasim} existence of
stably unbounded norms yields existence of quasi-morphisms.

\medskip
\noindent{\bf Proof of Propositions \ref{prop-gbar} and
\ref{thm-gbar}:}
 First of all note that every element $h \in \bar{G}$ can be
uniquely written in the following {\it normal form}: either $h
=(g_1,g_2)$ or $h=(g_1,g_2)t$. This readily yields Proposition
\ref{prop-gbar}. Second, we claim that it suffices to show that $G$
has a non-trivial homogeneous quasi-morphism, say $r$. Indeed, put
$\bar{r}(h) = r(g_1)+r(g_2)$, where $h$ is in the normal form as
above. A straightforward analysis shows that $\bar{r}$ is a (not
necessarily homogeneous!) quasi-morphism on $\bar{G}$. For instance,
if $h=(h_1,h_2)t$ and $f=(f_1,f_2)$ then $hf=(h_1f_2,h_2f_1)t$ and
hence
$$|\bar{r}(hf) -\bar{r}(h)-\bar{r}(f)| \leq
|r(h_1f_2)-r(h_1)-r(f_2)|+|r(h_2f_1)-r(h_2)-r(f_1)|$$ and hence is
uniformly bounded. The other cases are considered similarly. Finally
note that the stabilization $\bar{r}_{\infty}(h):=\lim_{n \to
\infty}\bar{r}(h^n)/n$ does not vanish on $h=(g,1)$ provided
$r(g)\neq 0$. Since $\bar{r}_{\infty}$ is a homogeneous
quasi-morphism, the claim follows.

Let $\nu$ be a stably unbounded norm on $\bar{G}$. Assume that
$\nu_{\infty}(w) >0$ for some $w \in \bar{G}$.

\medskip
\noindent{\sc Case 1:} $w=(g_1,g_2)$. Put $w_1 = (g_1,1)$ and
$w_2=(1,g_2)$. We claim that either $\nu_{\infty}(w_1)>0$ or
$\nu_{\infty}(w_2)>0$. Indeed, $w^k=w_1^kw_2^k$ and hence
$$0 < \nu_{\infty}(w) \leq \nu_{\infty}(w_1)+ \nu_{\infty}(w_2)\;,$$
which yields the claim.

\medskip
\noindent{\sc Case 2:} $w=(g_1,g_2)t$. Put $w_1 = (g_1g_2,1)$ and
$w_2=(1,g_2g_1)$. We claim that either $\nu_{\infty}(w_1)>0$ or
$\nu_{\infty}(w_2)>0$. Indeed, $w^{2k}=w_1^kw_2^k$ and hence
$$0 < \nu_{\infty}(w) \leq \frac{1}{2}(\nu_{\infty}(w_1)+ \nu_{\infty}(w_2))\;,$$
which yields the claim.

\medskip
\noindent Looking at elements $w_1$ and $tw_2t$ above we conclude
that there exists an element $u = (g,1)$ with $\nu_{\infty}(u) > 0$.
Replacing, if necessary, $u$ by its power we can assume that $g \in
G'$ (here we use that $H^1(G)$ is finite). Denote by $H \subset
\bar{G}$ the subgroup consisting of all elements of the form $(f,1)$
where $f \in G$. Clearly, $H$ is isomorphic to $G$ and $u \in G'$.
Furthermore, $t$ displaces $H$. Thus inequality \eqref{eq-master-1}
yields that
$$\nu(z) \leq 4\nu(t)\cdot cl_H(z)\;\;\forall z \in H'\;.$$
Substituting $z=u^k$, dividing by $k$ and passing to the limit as $k
\to \infty$ we get that
$$0 < \nu_{\infty}(u) \leq 4\nu(t)\cdot scl_H(u)\;.$$
Thus $scl_G(g) = scl_H(u)>0$. Therefore Bavard's theorem
\cite{Bavard} yields existence of a non-trivial homogeneous
quasi-morphism on $G$. \qed

\subsection{Inequalities with commutators}\label{sec-ineq-com}

Here we prove Theorem \ref{thm-master}(i). For an element $F \in G$,
we say that $g\in G$ is an {\it $F$-commutator} if $g=\Conj_f[F,h]$
for some $f,h\in G$. Note that the inverse of an $F$-commutator is
again an $F$-commutator.

Fix $F\in G$ such that the subgroups
$$H_0:= H,\;\; H_1:= \Conj_{F}H,\;\;...,\;\;H_{m} := \Conj_{F^{m}}H$$
pair-wise commute. We shall show that every element $x$ from the
commutator subgroup $H'$ with $cl_H(x)=m$ can be represented as a
product of seven $F$-commutators.  Note that given a
conjugation-invariant norm $\nu$ on $G$,  for every $F$-commutator
$g$ we have $\nu(g) \leq 2\nu(F)$. Thus we shall get that $\nu(x)
\leq 14\nu(F)$, which yields inequality \eqref{eq-master}.

We shall consider products $\prod_0^m \Conj_{F^i}(g_i)$, where
$g_i \in H$, $i =0,...,m$. Since $H_i$'s pair-wise commute, the
product of such elements $\prod_0^m \C_{F^i}(f_i)$ and $\prod_0^m
\Conj_{F^i}(g_i)$ can be computed component-wise: it equals
$\prod_0^m \Conj_{F^i}(f_ig_i)$.

\begin{lemma}
\label{rearrange-id} Let a collection of $g_i \in H$, $i=0,1,
\dots, m$ be such that $\prod_0^m g_i=1$. Then the product
$g=\prod_0^m \Conj_{F^i}(g_i)$ is an $F$-commutator.
\end{lemma}

\begin{proof}
We will show that $g=[F,\phi^{-1}]$ where $\phi=\prod_0^{m-1}
\C_{F^i}(\phi_i)$, $\{\phi_i\}_{i=0}^{m-1}$ is a collection of
elements of $H$ which will be defined later. We set $\phi_m=1$ for
convenience of notation.

Note that $[F,\phi^{-1}]=\C_F(\phi^{-1})\phi$ and $\C_F(\phi^{-1})$
equals the product $\prod_0^{m-1} \C_{F^{i+1}}(\phi_i^{-1}) =
\prod_1^m \C_{F^i}(\phi_{i-1}^{-1})$ whose terms lie in
$H_1,...,H_{m}$. Hence
$$
[F,\phi^{-1}] = \C_F(\phi^{-1})\phi = \phi_0\cdot \prod_1^m
\C_{F^i} (\phi_{i-1}^{-1}\phi_i)
$$
and the equation $[F,\phi^{-1}]=g$ is equivalent to the system
$$
\begin{cases}
\phi_0 = g_0 \\
\phi_0^{-1}\phi_1 = g_1 \\
\phi_1^{-1}\phi_2 = g_2 \\
\quad\dots \\
\phi_{m-1}^{-1}\phi_m = g_m \\
\end{cases}
$$
The solution of this system is $\phi_k=\prod_0^k g_i$,
$k=0,1,\dots,m$. The equation $\phi_m=1$ is satisfied by the
assumption $\prod g_i=1$.
\end{proof}

\begin{lemma}
\label{rearrange} Let $g_1,g_2,\dots,g_m$ be a collection of
elements of $H$. Then $g=\prod_m^1 g_i$ equals an $F$-commutator
times the product $\prod_1^m \C_{F^i}(g_i)$.
\end{lemma}

\begin{proof}
Introduce $g'_0=g$ and $g'_i=g_i^{-1}$. Note that $\prod_0^m
g'_i=1$. Then apply the previous lemma.
\end{proof}

\begin{lemma}\label{lem-two-Fcom}
Any commutator from $H$ is a product of two $F$-commutators.
\end{lemma}

\begin{proof}
Consider a commutator $[f,g]$ with $f,g \in H$. Then by Lemma
\ref{rearrange-id}, the elements
$$
(fg) \C_F(g^{-1}) \C_{F^2}(f^{-1})
$$
and
$$
(f^{-1}g^{-1}) \C_F(g) \C_{F^2}(f)
$$
are $F$-commutators. Their product is $[f,g]$.
\end{proof}

\medskip
\noindent {\bf End of the proof of Theorem \ref{thm-master}(i):}
Consider $h=\prod_m^1 [f_i,g_i]$ with $f_i,g_i \in H$. By Lemma
\ref{rearrange}, $h$ equals an $F$-commutator times a product
$\theta:= \prod_1^m \C_{F^i}([f_i,g_i])$. The latter in its turn is
equal to the commutator of two products $\phi:= \prod_1^m
\C_{F^i}(f_i)$ and $\psi:= \prod_1^m \C_{F^i}(g_i)$ since the
subgroups $H_i$ and $H_j$ commute for $i\ne j$. This proves
inequality \eqref{eq-master-com}.

Applying again Lemma \ref{rearrange} we have that $\phi=fx$ and
$\psi=gy$ where $f = f_m...f_1$ and $g=g_m...g_1$ and $x,y$ are
$F$-commutators. We write
$$\theta=[fx,gy]=[f,g]\cdot \Conj_g \{ \Conj_f(g^{-1}xg \cdot y \cdot
x^{-1}) \cdot y^{-1}\}\;.$$  Since $f,g \in H$, we have by Lemma
\ref{lem-two-Fcom} that $[f,g]$ equals a product of two
$F$-commutators. Hence $\theta$ is a product of six $F$-commutators
and therefore $h$ is a product of seven $F$-commutators. As we
explained in the beginning of this section, this completes the proof
of the theorem.\qed

\subsection{Packing and distortion of subgroups}

Let $G$ be a group and $H\subset G$ a subgroup. Consider the
embedding of metric spaces $(H', cl_H) \mapsto (G',cl_G)$.
Obviously $cl_G(w) \leq cl_H(w)$ for all $w \in H'$. It turns out
that, after stabilization, this inequality can be refined provided
$H$ is $m$-displaceable in $G$: the larger $m$ is, the stronger
$H'$ is distorted in $G'$ with respect to the stable commutator
lengths.

\medskip
\noindent
\begin{thm} \label{thm-algpack}
$$scl_G (w) \leq \frac{1}{p(G,H)} scl_H (w) \;\; \forall w \in
H'\;.$$
\end{thm}

\medskip
\noindent
\begin{example} \label{ex-packing} {\rm
Let $G=\widetilde{Sp(2n,R)}$ be the universal cover of the linear
symplectic group and let $H =\widetilde{Sp(2,R)} \subset G$. Here we
fix the splitting $\R^{2n} = \R^2\oplus\R^{2n-2}$. The monomorphism
$Sp(2,\R) \to Sp(2n,\R)$ which sends a matrix $A$ to $A \oplus
\id_{2n-2}$ induces the isomorphism of the fundamental groups
$\pi_1(Sp(2,\R))=\pi_1(Sp(2n,\R))=\Z$, and hence $H$ naturally
embeds into $G$. Let $(p_1,q_1,...,p_n,q_n)$ be the standard
symplectic coordinates on $\R^{2n}$. Denote by $I_j$ the symplectic
transformation which permutes $(p_1,q_1)$ and
$(p_j,q_j)$-coordinates. Write $\widetilde{I}_j$ for a lift of $I_j$
to $G$. Then the subgroups $\Conj_{I_j}(H)$ pairwise commute, and
hence $p(G,H) \geq n$. Denote by $e \in H$ the generator of the
center of $H$. One can show (see Remark \ref{rem-BIW} below) that
\begin{equation}\label{eq-comnorm-SL}
scl_H(e) = n\cdot scl_G(e)\;.
\end{equation}
 Thus the inequality in Theorem
\ref{thm-algpack} yields $p(G,H)\leq n$. We conclude that $p(G,H) =
n$ and the inequality is sharp. }
\end{example}

\medskip
\noindent \begin{example}{\rm Let $(M,\omega)$ be a symplectic
manifold, and let $U \subset M$ be an open subset. Let
$G=\Ham(M,\Omega)$ and let $H=\Ham(U,\omega)$. In this case the
algebraic packing number $p(G,H)$ has a simple geometric meaning:
It equals to the {\it geometric packing number} $p_{geom}(M,U)$
which is defined as the minimal number of diffeomorphisms from $G$
which take $U$ to pairwise disjoint subsets of $M$. In the case
when $U$ is a standard symplectic ball the geometric packing
number was intensively studied in the framework of the symplectic
packing problem (see \cite{Bi} for a survey). For instance, assume
that $M$ and $U$ are $2n$-dimensional symplectic balls. In the
case $n=1$ the geometric packing number is simply the integer part
of the ratio of the areas. In the case $n=2$ the situation is more
complicated: For instance, if the ratio of volumes of $M$ and $U$
lies in the interval $(8;(1+1/288)\cdot 8)$, the geometric packing
number equals $7$ (see \cite{MP}). It would be interesting to
explore the sharpness of the inequality in Theorem
\ref{thm-algpack} in these examples.}\end{example}

\medskip
\noindent The proof of Theorem \ref{thm-algpack} is based on the
following observation (thanks to Sasha Furman for help). For a
subgroup $H \subset G$ write $Q(H)$ for the set of homogeneous
quasi-morphisms on $H$ modulo morphisms, and for $\phi \in Q(H)$
put $$||\phi||_H =\sup_{x,y \in H} \phi([x,y])\;.$$

\medskip
\noindent
\begin{prop}\label{prop-1} Let $H_1,...,H_N$ be subgroups
of $G$ so that $H_i$ and $H_j$ commute for $i\neq j$. Put
$K=H_1\cdot ...\cdot H_N$. Then for every $\phi \in Q(K)$
$$||\phi||_K = \sum_{i=1}^N ||\phi||_{H_i}\;.$$
\end{prop}

\medskip
\noindent {\bf Proof of Proposition \ref{prop-1}:} Take any $x,y
\in H$ and write
$$x =x_1\cdot ... \cdot x_N,\; y = y_1 \cdot ... \cdot y_N \;,$$
where $x_i,y_i \in H_i$. Then
$$[x,y] = [x_1,y_1]\cdot ... \cdot [x_N,y_N]\;.$$
Since the commutators in the right hand side pair-wise commute we
get that for every quasi-morphism $\phi \in Q(K)$
$$\phi([x,y]) = \sum_{i=1}^N \phi([x_i,y_i])\;.$$
Since pairs $x_i,y_i$ can be chosen in an arbitrary way we get the
desired equality. \qed

\medskip
\noindent{\bf Proof of Theorem \ref{thm-algpack}:} Suppose that
$p(G,H)\geq N$. Then there exist elements $g_1=1,g_2,...,g_N$ so
that subgroups $H_i:=g_iHg_i^{-1}$ pair-wise commute. For every
$\phi \in Q(G)$ we have $||\phi||_{H_i} = ||\phi||_H$. Put $K
=H_1\cdot ...\cdot H_N$. Applying Proposition \ref{prop-1} we have
\begin{equation}\label{eq-1}
||\phi||_G \geq ||\phi||_K = N||\phi||_H\;.
\end{equation}
Denote by $Q_*(H)$ the set of non-vanishing quasi-morphisms from
$Q(H)$, and by $Q_*(G,H)$ the set of quasi-morphisms from $Q_*(G)$
which restrict to a non-vanishing quasi-morphism on $H$. Apply now
Bavard's theorem \cite{Bavard}: given $w \in H'$ we have
$$scl_H(w) = \frac{1}{2}\sup_{\phi \in Q_*(H)}
\frac{\phi(w)}{||\phi||_H} \geq \frac{1}{2}\sup_{\phi \in
Q_*(G,H)} \frac{\phi(w)}{||\phi||_H}\;.$$ Using inequality
\eqref{eq-1} above and applying the same Bavard's theorem we have
\begin{equation}
\label{eq-2} scl_H(w) \geq N \cdot\frac{1}{2}\sup_{\phi \in
Q_*(G,H)} \frac{\phi(w)}{||\phi||_G} = N
\cdot\frac{1}{2}\sup_{\phi \in Q_*(G)} \frac{\phi(w)}{||\phi||_G}=
Nscl_G(w)\;.
\end{equation}
The equality in the middle follows from the fact that for $\phi
\in Q_*(G)\setminus Q_*(G,H)$ and $w \in H'$ one has $\phi(w) =
0$. Using inequality \eqref{eq-2}, we readily complete the proof.
\qed

\medskip
\noindent\begin{rem}\label{rem-BIW} {\rm Denote by $G_n$ the
universal cover of the group $Sp(2n,\R)$ and by $e_n \in G_n$ the
generator of $\pi_1(Sp(2n,\R))$ with  Maslov index $2$. The group
$G_n$ carries unique homogeneous quasi-morphism $\mu_n$ with
$\mu_n(e_n)=1$ (see \cite{Barge-Ghys}). Put
$$I_n := \frac{||\mu_n||_{G_n}}{||\mu_1||_{G_1}}\;.$$
One can show that $I_n=n$. The only known to us proof of this
innocently looking fact is surprisingly involved: it can be
extracted from \cite{BIW} (thanks to A.~Iozzi and A.~Wienhard for
illuminating consultations). By the above-cited theorem due to
Bavard
$$\frac{scl_{G_1}(e_1)}{scl_{G_n}(e_n)}=I_n\;,$$
which proves equality \eqref{eq-comnorm-SL} above. }\end{rem}

\section{Topological arguments }\label{sec-topo}

\subsection{Portable manifolds}\label{subsec-to-port}

Let $M$ be a portable manifold. We shall use notations of
Definition \ref{def-portable}.

\begin{lemma}\label{lem-port-1}
There exists a neighborhood $U$ of the core $M_0$ of $M$ and a
diffeomorphism $\phi \in \Diff_0(M)$ so that the sets $\phi^i(U)$,
$i \geq 1$ are pair-wise disjoint.
\end{lemma}

\begin{proof} Choose a sufficiently small neighbourhood
$U$ of the core so that $\theta(U) \cap \text{Closure}(U) =
\emptyset$. Put $V = \theta(U)$ and consider the vector field $Y =
\theta_*X$ on $M$. Note that $V$ is an attractor of $Y$. In
particular there exists $\tau>0$ large enough so that  the closure
of $Y^{\tau}(U \cup V)$ is contained in $V$. Cutting off
$Y^{\tau}$ outside a sufficiently large compact set, we get that
there exists a diffeomorphism $\phi \in \Diff_0(M)$ so that
$$\text{Closure}\;\phi(U \cup V) \subset V\;.$$
Observe that $\phi^i(U) \subset \phi^{i-1}(V) \setminus
\phi^i(V)$. Thus the sets $\phi^i(U)$, $i \geq 1$ are pair-wise
disjoint.
\end{proof}

\medskip
\noindent{\bf Proof of Theorem \ref{thm-portable}}: Let $\nu$ be
any conjugation-invariant norm on $\Diff_0(M)$. It suffices to show that
$\nu$ is bounded.

 We shall use
notations of Definition \ref{def-portable} of a portable manifold.
Look at the neighborhood $U$ of the core and at the diffeomorphism
$\phi$ from Lemma \ref{lem-port-1}. Note that $\phi$ $m$-displaces
the subgroup $\Diff_0(U)$ for any $m$. Take any diffeomorphism $h
\in \Diff_0(U)$. Since the group $\Diff_0(U)$ is perfect, it follows
from inequality \eqref{eq-master-1} that $\nu(h) \leq 14\nu(\phi)$.

Further, take any diffeomorphism $f \in \Diff_0(M)$. The first
item of the Definition \ref{def-portable} guarantees that for
$\tau>0$ large enough $X^{\tau}(\supp f) \subset U$. Applying the
ambient isotopy theorem, we can find a diffeomorphism $\psi \in
\Diff_0(M)$ with $\psi(\supp f) \subset U$. Thus $\psi f
\psi^{-1}$ lies in $\Diff_0(U)$. We conclude that
$$\nu(f) = \nu(\psi f
\psi^{-1} )\leq 14 \nu(\phi)$$ which implies that $\nu$ is bounded.
This completes the proof. \qed

\medskip
\noindent{\bf Proof of Theorem \ref{thm-uni-perf-port}}: The proof
above shows that the diffeomorphism $\phi$ $m$-displaces the
subgroup $H:=\Diff_0(U)$ for any $m$. Corollary \ref{cor-port} above
implies that $cl_G(h) \leq 2$ for all $h \in \Diff_0(U)$, where
$G=\Diff_0(M)$. But every element $f \in G$ is conjugate to an
element from $H$. Thus $cld(M) \leq 2$. \qed

\medskip
\noindent \begin{rem}{\rm Theorem \ref{thm-portable} admits the
following straightforward  generalization. Let $G$ be any group
acting by homeomorphisms on a topological space $X$. Assume that
there exist two disjoint open subsets $U,V \subset X$ and an element
$\phi \in G$ which satisfy the following two easily verifiable
properties:
\begin{itemize}
\item[{(i)}] $\text{Closure}\;\phi(U \cup V) \subset V\;;$
\item[{(ii)}] For every finite collection of elements
$\psi_1,...,\psi_k \subset G$ there exists $h \in G$ so that
$$h\Big{(}\bigcup_{i=1}^k \text{support}(\psi_i)\Big{)} \subset U\;.$$
\end{itemize}
Then the commutator group $G'$ is bounded. }
\end{rem}

\subsection{Spheres} \label{subsec-to-sph}

\begin{lemma}\label{lem-discs}
Every diffeomorphism $f \in \Diff_0(S^n)$ can  be written as
$f=gh$ where $g \in \Diff_0(S^n \setminus \{z\})$ and $h \in
\Diff_0(S^n \setminus \{w\})$ for some points $z,w \in S^n$.
\end{lemma}

\medskip
\noindent Since $S^n \setminus \{\text{point}\}=\R^n$ is a
portable manifold, Theorem \ref{thm-1}(ii) follows from Theorem
\ref{thm-portable} and Theorem \ref{thm-uni-perf}(i) follows from
Theorem \ref{thm-uni-perf-port}.

\medskip
\noindent{\bf Proof of Lemma \ref{lem-discs}:} This fact is
standard: Let $\{f_t\}$, $t \in [0;1]$ be a path in $\Diff_0(S^n)$
with $f_0=\id$ and $f_1=f$. Choose a sufficiently small closed disc
$D \subset S^n$ so that $X:= \bigcup_t f_t(D) \neq S^n$. Pick a
point $z \notin X$. Since $S^n \setminus \{z\}$ is diffeomorphic to
$\R^n$, there exists a path $\{g_t\}$ of diffeomorphisms from
$\Diff_0(S^n \setminus \{z\})$ such that $g_0= \id, g_t|_D=f_t|_D$.
Pick a point $w$ in the interior of $D$. Note the path
$\{g_t^{-1}f_t\}$ is compactly supported in $S^n \setminus \{w\}$.
Thus the diffeomorphisms $g:=g_1$ and $h:= g^{-1}f$ are as required
in the lemma. \qed

\subsection{Three-manifolds}\label{subsec-to-threeman}

Here we prove Theorem \ref{thm-1}(iii). By a {\it graph} in a
manifold we mean a piecewise smoothly embedded graph. By a smooth
isotopy of a graph we mean an isotopy which extends to a smooth
isotopy of its tubular neighborhood. We shall use without a
special mentioning the following fact (see e.g. \cite{Luft}): any
smooth compactly supported diffeomorphism $\phi$ of an open
handlebody $U$ is isotopic to the identity through compactly
supported diffeomorphisms, that is $f \in \Diff_0(U)$.

\medskip
\noindent
\begin{lemma}[Fundamental Lemma]\label{lem-fund} Let $\Gamma$ and $K$ be two disjoint graphs and $M$.
Let $f_t: \Gamma \to M$, $t \in [0;1]$  be a smooth isotopy with
$f_0|_{\Gamma}=\id$ and $f_1(\Gamma)\cap K = \emptyset$. Then
there exist a diffeomorphism $h$ of $M$ supported in a ball and a
diffeomorphism $\phi \in \Diff_0(M \setminus K)$ so that
$$f_1|_{\Gamma}=h \circ \phi|_{\Gamma}\;.$$
\end{lemma}

\medskip
\noindent Let us prove the theorem assuming the lemma.

\medskip
\noindent {\bf Proof of Theorem \ref{thm-1}(iii):} Take any norm
$\nu$ on $\Diff_0(M)$. A graph is called {\it the Heegard graph} if
its complement is diffeomorphic to an open handlebody. Every
three-manifold contains a Heegard graph (for instance, a
neighborhood of the 1-skeleton of a triangulation of $M$). Choose a
pair of disjoint Heegard graphs $L$ and $K$ in $M$. Fix a
sufficiently small tubular neighborhood $U$ of $L$. Since
$U,M\setminus K$ and $M \setminus L$ are open handlebodies and
therefore are portable, Theorem \ref{thm-portable} implies that the
norm $\nu$, when restricted to $\Diff_0$ of these submanifolds, does
not exceed some constant $C>0$. We shall assume also that the same
inequality holds for the restriction of $\nu$ to $\Diff_0$ of any
ball in $M$ (we use here that all balls are pair-wise isotopic and
portable).

We shall show that
\begin{equation}\label{eq-3mds}
\nu(f) \leq 5C
\end{equation}
for every $f \in \Diff_0(M)$ with $f(U) \cap K = \emptyset$. Note
that this yields the same inequality for {\it every}  $f$. Indeed,
perturbing $K$ to $K'$ by a small ambient isotopy of $M$ and
shrinking $U$ to $U'$ by an ambient isotopy of $M$ we can always
achieve that $f(U') \cap K' =\emptyset$. But the subgroups
$\Diff_0(U')$ and $\Diff_0(M \setminus K')$ are conjugate in
$\Diff_0(M)$ to $\Diff_0(U)$ and $\Diff_0(M \setminus K)$
respectively, and hence the restriction of the norm $\nu$ to these
subgroups is bounded by the same constant $C$ which yields
inequality \eqref{eq-3mds}. From now on we assume that $f(U) \cap K
= \emptyset$.

Let $N \subset U\setminus L$ be any embedded graph so that the
induced homomorphism $\pi_1(N) \to \pi_1(U\setminus L)$ is an
isomorphism. Put $\Gamma = L \cup N$, and apply the Fundamental
Lemma. We get a diffeomorphism $h$ supported in a ball, and a
diffeomorphism $\phi \in \Diff_0(M\setminus K)$ so that
$f{|}_{\Gamma}=h \circ \phi|_{\Gamma}\;.$ Denote
$\psi=(h\phi)^{-1}f$ and observe that $\psi{|}_{\Gamma}=\id$.

In particular, $\psi$ fixes $L$. We wish to correct $\psi$ and get a
diffeomorphism fixing {\it a neighborhood} of $L$. This is the point
where the graph $N$ enters the play. More precisely, we claim that
there exist diffeomorphisms $\xi,\theta \in \Diff_0(U)$ and $\eta
\in \Diff_0(M \setminus L)$ so that $\psi= \xi\eta\theta$. Indeed,
since $\psi$ fixes $L$, there exists a sufficiently small tubular
neighborhood $V \subset U$ of $L$ and a diffeomorphism $\theta \in
\Diff_0(U)$ so that $\psi\theta^{-1}(V)=V$. Put $\tau:=
\psi\theta^{-1}$. Since $U \setminus L$ retracts to $\partial V$ and
$\psi$ fixes $N$ we conclude that $\tau$ induces the identity
isomorphism of $\pi_1(\partial V)$. It is well known (see e.g.
\cite{Luft}) that therefore $\tau{|}_V:V \to V$ is isotopic to the
identity. Hence there exists a diffeomorphism $\xi \in \Diff_0(U)$
which coincides with $\tau$ on $V$, and so $\eta:= \xi^{-1}\tau$ is
supported in $M \setminus L$. The claim follows.

Finally, write
$$f =h\phi\psi= h\phi\xi\eta\theta\;.$$
Note that $h \in \Diff_0(B)$ where $B$ is a ball, and hence $\nu(h)
\leq C$ where the constant $C$ was chosen in the beginning of the
proof. Furthermore, $\phi \in \Diff_0(M \setminus K)$, $\xi,\theta
\in \Diff_0(U)$ and $\eta \in \Diff_0(M \setminus L)$. Thus $\nu(f)
\leq 5C$ which proves inequality \eqref{eq-3mds}. This completes the
proof. \qed

\medskip
\noindent {\bf Proof of Theorem \ref{thm-uni-perf}(i):} In the proof
above we represented every diffeomorphism from $\Diff_0$ of a closed
connected three-manifold $M$ as a product of $5$ diffeomorphisms
from $\Diff_0$ of portable manifolds. Applying Theorem
\ref{thm-uni-perf-port} we get the desired estimate $cld(M)\leq 10$.
\qed

\medskip
\noindent {\bf Proof of Lemma \ref{lem-fund}:} The proof is
divided into several steps.

\medskip
\noindent{\sc Step 1:} Let $\Gamma, L \subset M$ be disjoint
embedded graphs, and $f_t:\Gamma \to M$ be a smooth isotopy. Put
$\Gamma_t:=f_t(\Gamma)$. We say that the crossing point
$y=f_{\tau}(x) \in \Gamma_{\tau} \cap L$ is {\it generic} if the
points $x$ and $y$ lie in smooth interior parts of $\Gamma$ and
$L$ respectively and
$$f_{\tau*}(T_x\Gamma)\oplus T_yL \oplus \R \cdot
\frac{\partial}{\partial t}{\Big |}_{t=\tau} f_tx = T_yM\;.$$
Introduce two modifications of the isotopy $f_t$ at a generic
crossing point.

\medskip
\noindent{\sc Type I modification (removing the crossing point):}
Here we assume that $L$ is {\bf a segment} with the endpoints $A$
and $B$ and $y= \Gamma_{\tau} \cap L$ is a generic crossing point.
Choose $\epsilon >0$ small enough so that $y$ is the only crossing
point on the time interval $I:= [\tau-\epsilon; \tau+\epsilon]$.
Choose a sufficiently small neighborhood $U$ of $L$. Let $h_s, s
\in I$ be a path in $\Diff_0(U)$ so that $h_s =\id$ outside a
small neighborhood of $s=\tau$, $h_s(L) \subset L$ and $h_s(B) =
B$ for all $s$, and  $h_{\tau}$ shrinks $L$ so that $y \notin
h_\tau (L)$. Replace the piece $\{\Gamma_t\}_{t \in I} $ of the
original isotopy by $\{\Gamma'_t\}_{t \in I}$ where $\Gamma'_t =
h_{t}^{-1}\Gamma_t$. Note that $\Gamma_t \cap h_{t}(L) =
\emptyset$, and hence $\Gamma'_t \cap L = \emptyset$, for all $t
\in I$.

\medskip
\noindent{\sc Type II modification (decomposition):} Here $\Gamma$
and $L$ are arbitrary graphs, and $y = f_{\tau}(x) \in
\Gamma_{\tau}\cap L$ is a generic crossing point. Choose $\epsilon
>0$ small enough so that $y$ is the only crossing point on the
time interval $I:= [\tau-2\epsilon; \tau+2\epsilon]$. There exists
a neighborhood $E$ of $y$ diffeomorphic to a Euclidean cube
$$Q =\{(u,v,w) \in \R^3 \;\Big{|}\;\; |u|,|v|,|w| < 2\epsilon\}\;$$ so that
$L \cap Q$ is the vertical segment $\{u=v=0,w\in
[-2\epsilon;2\epsilon]\}$ and $\Gamma_t \cap Q$ is the segment
$c_{t-\tau}:=\{u =t-\tau, v \in [-2\epsilon;2\epsilon], w =0\}$
for $t \in I$. Thus the isotopy $\Gamma_t$ inside $Q$ is given by
the motion of the segment $c_{-2\epsilon}$ in the $(u,v)$-plane in
the direction of the $u$-axis. In this picture, the crossing point
$y$ is the origin.

Let us agree on the following wording: Suppose that two curves
$\alpha_0$ and $\alpha_1$ in the $(u,v)$-plane are given by the
graphs $\{u =F_0(v)\}$ and $\{u = F_1(v)\}$ of smooth functions
$F_0,F_1:[-2\epsilon;2\epsilon]\to \R$. The {\it linear isotopy}
between $\alpha_0$ and $\alpha_1$ is formed by graphs of $(1-s)F_0
+sF_1$, $s \in [0;1]$.

The modification we are going to describe is local. Fix a smooth
cut-off function $\rho:[-2\epsilon;2\epsilon]\to [0;3\epsilon/2]$
which is supported in a very small neighborhood of $0$ and which
satisfies $\rho(0)=3\epsilon/2$. Denote $\beta^{\pm} = c_{\pm
\epsilon}$. Consider the curve
$$\alpha= \{u = -\epsilon + \rho(v), v \in [-2\epsilon;2\epsilon],w=0\}\;.$$
Modify the original isotopy on the time interval $I':=
[\tau-\epsilon;\tau+\epsilon]$ as follows: first make a linear
isotopy from $\beta^-$ to $\alpha$, and then a linear isotopy from
$\alpha$ to $\beta^+$. We extend the curves appearing in the
process of this isotopy outside $Q$ by appropriate $\Gamma_t$'s
and make an obvious change of time in order to fit into the time
interval $I'$.

The following features of the modified isotopy are crucial for our
further purposes. The isotopy from $\beta^-$ to $\alpha$ can be
realized by an isotopy of diffeomorphisms of $M$ supported in a
ball $B \subset Q$. The isotopy from $\alpha$ to $\beta^+$ does
not hit $L$ and hence can be extended to an ambient isotopy of $M$
which is fixed near $L$.

\medskip
\noindent {\sc Step 2:} After these preliminaries, we pass to the
situation described in the formulation of the lemma: Let $\Gamma,K$
be two disjoint graphs in $M$ and let $f_t:\Gamma \to M$, $t \in
[0;1]$ be a smooth isotopy with $f_1(\Gamma) \cap K = \emptyset$.
After a small perturbation of the isotopy with fixed end points we
can assume that the following conditions hold:

\medskip
\noindent (C1) The set $$\{(x,t) \in \Gamma \times
[0;1]\;\Big{|}\;\; f_t(x) \in K\}$$ consists of $N$ pairs
$(x_i,t_i)$, $i=1,...,N$ so that $\{x_i\}$ are distinct points of
$\Gamma$, $0<t_1<...<t_N<1$ and $y_i=f_{t_i}(x)$ are distinct
generic crossing points.

\medskip
\noindent (C2) The curves $\gamma_i:= \{f_t(x_i)\}_{t \in [0;1]}$
are pairwise disjoint embedded segments.

\medskip
\noindent (C3) For each $i$, the isotopy $f_t: \Gamma\setminus
\{x_i\} \to M$ crosses $\gamma_i$ generically.

\medskip
\noindent We shall remove the latter crossings using the Type I
modification (see Step 1): Note that each such crossing occurs in
the subsegment of $\gamma_i$ which is either of the form
$[x_i;f_{t_i-\delta}x_i]$ or $[f_{t_i+\delta}x_i;f_1x_i]$, where
$\delta>0$ is small enough. We apply Type I modification to these
segments keeping the end point $f_{t_i\pm \delta}x_i$ fixed (such an
end point is denoted by $B$ in the local description of a Type I
modification above). Note that each such modification is localized
near some $\gamma_i$ and hence does not create new crossings, so the
process stops after a finite number of modifications. Thus we
replace assumption (C3) above by a stronger one:

\medskip
\noindent (C3')  For each $i$, the isotopy $f_t: \Gamma\setminus
\{x_i\} \to M$ does not hit $\gamma_i$.

\medskip
\noindent {\sc Step 3:} It would be convenient to make a change of
time in our isotopy as follows. We assume that $f_t$ is defined on
the time interval $t \in [0;N+1]$ and the crossings times are
consecutive integers $t_i =i$, $i=1,...,N$. Assumptions (C1) and
(C2) of the previous step yield existence of embedded pair-wise
disjoint parallelepipeds $P_i \subset M$, $i=1,...,N$ (each
parallelepiped $P_i$ is a neighborhood of the segment $\gamma_i$)
equipped with local coordinates $u \in [-1;N+2], v \in [-1;1], w \in
[-1;1]$ so that the following holds:
$$\gamma_i = \{(u,0,0)\;\Big{|}\;\; u \in [0;N+1]\}\;,$$
$$ K \cap P_i = \{(i,0,w)\;\Big{|}\;\; w \in [-1;1]\},\;\;\; \Gamma \cap P_i= \{(0,v,0)\;\Big{|}\;\; v \in
[-1;1]\},\;$$ and $$f_t(0,v,0) = (t,v,0) \;\;\forall t \in [0;N+1],v
\in [-1;1]\;.$$ In addition, assumption (C3') of the previous step
guarantees that $P_i$'s can be chosen so thin that
\begin{equation}\label{eq-tails}
f_t(\Gamma \setminus P_i) \cap P_i = \emptyset \;\;\forall t \in
[0;N+1]\;.
\end{equation}

\medskip
\noindent {\sc Step 4:} Let $Q_i \subset P_i$ be a sufficiently
small cube centered at the crossing $(i,0,0)$ whose edges have the
length $4\epsilon$ and are parallel to the coordinate axes.
Perform a Type II modification of our isotopy inside $Q_i$: We
keep notations $\alpha_i,\beta^{\pm}_i$ (with the extra sub-index
$i$) for special curves appearing in the description of the
modification presented in Step 1. The reader should have in mind
that the current $u$-coordinate is shifted by $i$ in comparison to
the one of Step 1, and the crossing time $\tau$ equals $i$.

Thus we assume that
$$\beta^{\pm}_i =\{(i\pm\epsilon, v,0)\;\Big{|}\;\; v \in
[-2\epsilon,2\epsilon]\;\}\;.$$ Set
$$\Gamma_i^{-}=f_{i-\epsilon}(\Gamma)\;\;\;\text{and}\;\;\;
\Gamma_i^+ = (f_{i-\epsilon}(\Gamma)\setminus \beta_i^{-})\cup
\alpha_i \;, \;\; i =1,...,N\;.$$  Note that $\Gamma_i^+ =
h_i(\Gamma_i^-)$, where $h_i \in \Diff_0(B_i)$ and $B_i \subset
P_i$ is a ball.

It will be convenient to put $\Gamma_0^+ =\Gamma$ and
$\Gamma_{N+1}^- = f_1(\Gamma)$. Recall that we write $\Gamma_t
=f_t(\Gamma)$.

\medskip
\noindent{\sc Step 5:} Fix $i \in \{0;...;N\}$. Let us focus on
the following isotopy taking $\Gamma_i^+$ to $\Gamma^-_{i+1}$ : we
proceed according to the description of the Type II modification
(see Step 1) until we reach the graph $\Gamma_{i+\epsilon}$ which
extends $\beta^+_i$ (this move is empty when $i=0$), and then move
on with the original isotopy $f_t$ until $\Gamma^-_{i+1}$. Note
that this isotopy does not hit $K$. Furthermore, the
(time-dependent) vector field $\zeta^{(i)}_t$ of this isotopy,
which is defined along the image of $\Gamma_i^+$ at the time
moment $t$, is parallel to the $u$-axis {\it in each of the
parallelepipeds} $P_j$, $j=1,...,N$. Now we shall use property
\eqref{eq-tails} of the original isotopy: It guarantees that one
can cut off $\zeta^{(i)}_t$ near $K$ and extend it to the whole
$M$ so that {\it it remains parallel to the $u$-axis in all
$P_j$'s.} After such an extension we get an isotopy  supported in
$M \setminus K$ so that its time-1-map $\phi_i$ sends $\Gamma_i^+$
to $\Gamma^-_{i+1}$.

The following property of maps $\phi_i$, which readily follows
from the above discussion on vector fields $\zeta^{(i)}_t$, is
crucial for the final step of the proof:
\begin{equation}\label{eq-fin-step}
\phi_N \circ ...\circ \phi_i (B_i) \subset P_i \;\;\forall
i=1,...,N\;.
\end{equation}

\medskip
\noindent {\sc Step 6:} We have
\begin{equation}\label{eq-decomp} f_1|_{\Gamma} = \phi_Nh_N\circ ...\circ
\phi_1h_1\phi_0|_{\Gamma}\;,\end{equation}
where the
diffeomorphisms $h_i \in \Diff_0(B_i)$ and the balls $B_i$ appear
in Step 4, and the diffeomorphisms $\phi_i \in \Diff_0 (M
\setminus K)$ are constructed in the previous step. Put
$$g_i = (\phi_N \circ ...\circ \phi_i)h_i(\phi_N \circ ...\circ
\phi_i)^{-1}\;,\; i=1,...,N\;.$$ Note that $g_i \in \Diff_0(B'_i)$
where $B'_i= \phi_N \circ ...\circ \phi_i (B_i)\;.$ By
\eqref{eq-fin-step}, the balls $B'_i$ are pair-wise disjoint, and
hence the diffeomorphism $h:= g_N \circ ... \circ g_1$ is also
supported in a ball. Finally, put $$\phi=\phi_N \circ ...\circ
\phi_0 \in \Diff_0(M \setminus K)$$ and observe that in view of
equation \eqref{eq-decomp} $f_1|_{\Gamma} = h\phi|_{\Gamma}$. This
finishes off the proof of the lemma. \qed

\bigskip
\noindent {\bf Acknowledgments.} We thank A.~Furman, D.~Calegari,
E.~Ghys, A.~Iozzi, N.~Ivanov, D.~Kotschick, D.~McDuff, A.~Muranov,
Y.~Shalom and A.~Wienhard for useful discussions.

\bigskip

\noindent
\begin{tabular}{l}
Dmitri Burago \\
Department of Mathematics\\
The Pennsylvania State University, \\
University Park, PA 16802, U.S.A. \\
burago@math.psu.edu \\
\end{tabular}

\bigskip
\noindent
\begin{tabular}{l}
Sergei Ivanov \\
Steklov Math. Institute,\\
St.~Petersburg, Russia\\
svivanov@pdmi.ras.ru \\
\end{tabular}

\bigskip
\noindent
\begin{tabular}{l}
Leonid Polterovich \\
School of Mathematical Sciences \\
Tel Aviv University \\
Tel Aviv 69978, Israel \\
polterov@post.tau.ac.il \\
\end{tabular}

\end{document}